\documentclass[12pt]{amsart}
\usepackage{amssymb}
\usepackage{amsthm}
\usepackage{bm}
\usepackage{latexsym}
\usepackage{amsmath}
\usepackage{eufrak}
\usepackage{mathrsfs}
\usepackage{amscd}
\usepackage[all]{xy}
\usepackage[usenames]{color}
\usepackage[dvips]{graphicx}

\theoremstyle{plain}
\newtheorem{teo}{Theorem}[section]

\newtheorem{question}[teo]{Question}
\newtheorem{prop}[teo]{Proposition}
\theoremstyle{definition}
\newtheorem{defi}{Definition}[section]
\newtheorem{obs}{Remark}[section]
\newtheorem{example}{Example}[section]

   \DeclareMathOperator{\sing}{sing}

\def\K{{\mathbb{K}}}
\def\R{{\mathbb{R}}}

\def\A{{\mathcal{A}}}
\def\X{{\mathcal{X}}}

\def\B{{\mathscr{B}}}

\def\LL{{\mathscr{L}}}

\def\Z{{\mathbb Z}}
\def\Q{{\mathbb Q}}
\def\H{{\mathscr H}}
\def\PP{{\mathscr P}}
\def\P{{\mathbb P}}
\def\C{{\mathbb C}}
\def\Af{{\mathbb A}}

\begin{document}
\bibliographystyle{amsplain}

\title[Construcci\'on]{On line arrangements with applications to $3$-nets}
\author{Giancarlo Urz\'ua}

\email{urzua@math.umass.edu} \maketitle

\begin{abstract}
We show a one-to-one correspondence between arrangements of $d$ lines in $\P^2$, and lines in $\P^{d-2}$. We apply this correspondence to
classify $(3,q)$-nets over $\C$ for all $q\leq 6$. When $q=6$, we have twelve possible combinatorial cases, but we prove that only nine of them
are realizable over $\C$. This new case shows several new properties for $3$-nets: different dimensions for moduli, strict realization over
certain fields, etc. We also construct a three dimensional family of $(3,8)$-nets corresponding to the Quaternion group.
\end{abstract}

%---------------------------------------------------------------------------------------------------------------------------------------------
\section{Introduction.} \label{s1}

We start by showing a one-to-one correspondence between arrangements of $d$ lines in $\P^2$, and lines in $\P^{d-2}$. This correspondence is a
particular case of a more general one between arrangements of $d$ sections on ruled surfaces, which generalize line arrangements (see Remark
\ref{r0}), and certain curves in $\P^{d-2}$. This was developed by the author in \cite{Urzua4}. In there, we consider arrangements as single
curves via moduli spaces of pointed stable rational curves. One of the main ingredients is the description of these moduli spaces given by
Kapranov in \cite{KaVeronese93} and \cite{KaChow93}. In the present paper, we only treat the case of line arrangements, giving a proof of this
correspondence by means of quite elementary geometry. An important consequence of seeing an arrangement as a single curve is, as it turns out,
to answer questions about realization of certain line arrangements. That is the second part of this paper. Through this correspondence, we are
able to classify the so-called $(3,q)$-nets over $\C$ for all $q \leq 6$, and the Quaternion nets. This classification shows various new
properties for $3$-nets, and opens the question about realization of Latin squares in $\P_{\C}^2$ (they give the combinatorial data which
defines a $3$-net). The following is an outline of the paper.

In Section \ref{s2}, we prove a one-to-one correspondence between arrangements of $d$ lines, and lines in $\P^{d-2}$. An arrangement of $d$
lines is a set $$\A=\{L_1,L_2,\ldots,L_d \}$$ of $d$ labeled lines in $\P^2$ such that $\bigcap_{i=1}^d L_i=\emptyset$. As it was pointed out
above, our general correspondence is between arrangements of sections and curves. For this reason, instead of considering line arrangements, we
consider pairs $(\A,P)$ where $\A$ is an arrangement of $d$ lines, and $P \in \P^2$ is a point outside of $\bigcup_{i=1}^d L_i$. In Proposition
\ref{p1}, we prove a one-to-one correspondence between these pairs, up to projective equivalence, and lines in $\P^{d-2}$ outside of a fixed
hyperplane arrangement $\H_d$. We also sketch a second proof of it (see Remark \ref{r0}) to hint the more general correspondence mentioned
before (see \cite{Urzua4} for details).

Under this correspondence, for a fixed arrangement $\A$, different choices of $P$ may produce different lines in $\P^{d-2}$. Thus, if one fixes
the combinatorial type of $\A$ and wants to study its moduli space, the presence of this artificial point $P$ introduces more parameters than
needed. In practice, even the question of realization of $\A$ becomes hard with this extra point $P$. To eliminate this difficulty, we take $P$
in $\A$ and consider the new pair $(\A',P)$, where the lines in $\A'$ are the lines in $\A$ not containing $P$ (with a certain labelling). By
taking $P$ as a point lying on several lines of $\A$, one greatly simplifies computations to prove or disprove its realization over some field,
and to find a moduli space for its combinatorial type.

In Sections \ref{s3} and \ref{s4}, we use our method to study a particular type of line arrangements which are called nets. There is a large
body of literature about them (cf. \cite{Aczel}, \cite{BarStram}, \cite{Chein}, \cite{DeKe74}, \cite{LibYuz00}, \cite{YuzNets04},
\cite{Stipins1}, \cite{Stipins2}, \cite{FalkYuz07}). Nowadays, they are of interest to topologists who study resonance varieties of complex line
arrangements (see \cite{LibYuz00}, \cite{YuzNets04}, \cite{Buzu05}, \cite{FalkYuz07}). In general, they can be thought as the geometric
structures of finite quasigroups, which in turn are intimately related with Latin squares \cite{DeKe74}. We define $(p,q)$-nets in Section
\ref{s3}. We exemplify our correspondence by computing the Hesse arrangement, which is the only $(4,3)$-net over $\C$, and by showing that
$(4,4)$-nets do not exist in characteristic different than $2$ (they do in characteristic $2$, see Example \ref{e3}). In \cite{YuzNets04},
Yuzvinsky proved that $(p,q)$-nets over $\C$ are only possible for $p=3,4,5$ (not true in positive characteristic, where any $p$ is possible,
see Example \ref{e3}). Examples of $(5,q)$-nets were unknown, and for $(4,q)$-nets the only example was the Hesse arrangement. In
\cite{Stipins1}, Stipins proved that $(5,q)$-nets do not exist over $\C$, leaving open the case $p=4$. It is believed that the only $(4,q)$-net
is the Hesse arrangement. In Section \ref{s4}, we present a classification for $(3,q)$-nets over $\C$ with $q \leq 6$.

It is known that a $q \times q$ Latin square provides the combinatorial data which defines a $(3,q)$-net (see \cite{DeKe74}, \cite{LibYuz00},
\cite{kawa07}, \cite{Stipins1}). Until very recently, the only known $(3,q)$-nets corresponded to Latin squares coming from multiplication
tables of certain abelian groups. Yuzvinsky conjectured in \cite{YuzNets04} that this should always be the case. In \cite{Stipins1}, there was
given a three dimensional family of $(3,5)$-nets not coming from a group. For the case $q=6$, we have twelve possible cases associated to the
twelve main classes of $6 \times 6$ Latin squares. In Section \ref{s4}, we show that only nine of them are realizable in $\P^2$ over $\C$. These
nine cases present new properties for $3$-nets: we have four three dimensional and five two dimensional families, some of them define nets
strictly over $\C$, for others we have nets over $\R$ or even over $\Q$, etc. After that, we construct a three dimensional family of
$(3,8)$-nets associated to the Quaternion group, which has members defined over $\Q$. The new cases corresponding to the symmetric and
Quaternion groups show that there are $(3,q)$-nets associated to non-abelian groups (see \cite[Conj. 6.1]{YuzNets04}). Out of this, a natural
question is: \textit{find a combinatorial characterization of the main classes of Latin squares (see Remark \ref{r2}) which realize $(3,q)$-nets
in $\P_{\C}^2$}.

\vspace{0.1 cm} We denote the projective space of dimension $n$ by $\P^n$, and a point in it by $[x_1:\ldots:x_{n+1}]=[x_i]_{i=1}^{n+1}$. If
$P_1, \ldots,P_r$ are $r$ distinct points in $\P^n$, then $\langle P_1, \ldots,P_r \rangle$ is the projective linear space spanned by them. The
points $P_1, \ldots,P_{n+2}$ in $\P^{n}$ are said to be in general position if no $n+1$ of them lie in a hyperplane.

\vspace{0.1cm} \textbf{Acknowledgments}: I am grateful to Dave Anderson, my thesis advisor Igor Dolgachev, Sean Keel, Finn Knudsen, Janis
Stipins, and Jenia Tevelev for valuable discussions. I would also like to acknowledge the referee for helping me to improve the exposition of
this paper. \vspace{0.2cm}

%\tableofcontents

%--------------------------------------------------------------------------------------------------------------------------------------------
\section{Arrangements of $d$ lines in $\P^2$, and lines in $\P^{d-2}$.} \label{s2}

\begin{defi}
Let $d\geq 3$ be an integer. An \underline{arrangement of $d$ lines} $\A$ is a set of $d$ labeled lines $\{L_1, \ldots, L_d \}$ in $\P^2$ such
that $\bigcap_{i=1}^d L_i = \emptyset$. \label{d1}
\end{defi}

When the labelling is not relevant, we will consider $\A$ as the plane curve $\bigcup_{i=1}^d L_i$. We introduce ordered pairs $(\A,P)$, where
$\A$ is an arrangement of $d$ lines in $\P^2$, and $P$ is a point in $\P^2 \setminus \A$. If $(\A,P)$ and $(\A',P')$ are two such pairs, we say
that they are isomorphic if there exists an automorphism $T$ of $\P^2$ such that $T(L_i)=L'_i$ for every $i$, and $T(P)=P'$. Let $\LL_d$ be the
\underline{set of isomorphism classes of pairs} $(\A,P)$. For example, clearly $\LL_3$ is a set with only one element, represented by the class
of the pair $$\big( \{\{x=0\}, \{y=0 \}, \{z=0\} \}, [1:1:1] \big).$$

On the other hand, let us fix $d$ points in $\P^{d-2}$ in general position. We precisely take $P_1=[1:0:\ldots:0],
P_2=[0:1:0:\ldots:0],\ldots,P_{d-1}=[0:\ldots:0:1], P_d=[1:\ldots:1]$. Consider the projective linear spaces $$\Lambda_{i_1, \ldots,i_r}=
\langle P_j : j\notin \{i_1, \ldots,i_r \} \rangle,$$ where $1\leq r \leq d-1$ and $i_1, \ldots,i_r$ are distinct numbers, and let $\H_{d}$ be
the union of all the hyperplanes $\Lambda_{i,j}$. Hence, $\Lambda_{i,j}=\{ [x_1:\ldots:x_{d-1}] \in \P^{d-2}: x_i=x_j \}$ for $i,j\neq d$,
$\Lambda_{i,d}=\{ [x_1:\ldots:x_{d-1}] \in \P^{d-2}: x_i=0 \}$, and $$\H_d = \{ [x_1:\ldots:x_{d-1}] \in \P^{d-2}: \ x_1 x_2 \cdots x_{d-1}
\prod_{i<j} (x_j-x_i)=0 \}.$$

The proof of the following proposition is inspired by a particular case of the so-called Gelfand-MacPherson correspondence \cite[Chap.
2]{KaChow93}.

\begin{prop}
There is a one-to-one correspondence between $\LL_d$ and the set of lines in $\P^{d-2}$ not contained in $\H_d$. \label{p1}
\end{prop}

\begin{proof}
Let us fix a pair $(\A,P)$, where $\A$ is defined by the linear polynomials $$L_i(x,y,z)= a_{i,1} x+ a_{i,2} y+ a_{i,3} z, \ 1\leq i \leq d.$$

Consider the embedding $\iota_{(\A,P)}: \P^2 \hookrightarrow \P^{d-1}$ given by $$ [x:y:z] \mapsto
\Bigl[\frac{L_1(x,y,z)}{L_1(P)}:\ldots:\frac{L_d(x,y,z)}{L_d(P)}\Bigr].$$ Then, $\iota_{(\A,P)}(\P^2)$ is a projective plane,
$\iota_{(\A,P)}(P)=[1:\ldots:1]$, and $\iota_{(\A,P)}(L_i)=\iota_{(\A,P)}(\P^2) \cap \{y_i=0 \}$ for every $i \in \{1,2, \ldots,d \}$. We now
consider the projection $$\varrho: \P^{d-1} \setminus [1:\ldots:1] \rightarrow \P^{d-2}, \ \ \ [y_1:y_2:\ldots:y_d] \mapsto
[y_1-y_d:y_2-y_d:\ldots:y_{d-1}-y_d].$$ In this way, if $\Sigma_{i,j}=\{ [y_1:y_2:\ldots:y_d] : \ y_i=y_j \}$, we see that
$\varrho(\Sigma_{i,j})=\Lambda_{i,j}$. Therefore, we have that $\varrho \big(\iota_{(\A,P)}(\P^2) \big)$ is a line in $\P^{d-2}$ not contained
in $\H_d$. To show the one-to-one correspondence, we need to prove that $(\A,P) \mapsto \varrho \big(\iota_{(\A,P)}(\P^2) \big)$ gives a
well-defined bijection between $\LL_d$ and the set of lines in $\P^{d-2}$ not contained in $\H_d$. Clearly we have a bijection between
projective planes in $\P^{d-1}$ passing through $[1:\ldots:1]$ and not contained in $\bigcup_{i,j} \Sigma_{i,j}$, and the set of lines in
$\P^{d-2}$ not contained in $\H_d$.

Let $T: \P^2 \rightarrow \P^2$ be an automorphism of $\P^2$. Let $B=\big(b_{i,j} \big)$ be the $3 \times 3$ invertible matrix corresponding to
$T^{-1}$. Consider the pair $(\A',P')$ defined by $\A'=\{ L'_i=T(L_i) \}_{i=1}^d$ and $P'=T(P)$. Then, the equations defining the lines $L'_i$
are \begin{center} $ \big(\sum_{j=1}^3 a_{i,j}b_{j,1} \big)x+ \big(\sum_{j=1}^3 a_{i,j}b_{j,2} \big) y + \big(\sum_{j=1}^3 a_{i,j}b_{j,3} \big)
z=0.$ \end{center} Hence, we obtain that $\iota_{(\A,P)}= \iota_{(\A',P')} \circ T$, and so our map $(\A,P) \mapsto \varrho
\big(\iota_{(\A,P)}(\P^2) \big)$ is well-defined on $\LL_d$.

It is clearly surjective, so we only need injectivity. Let $\iota_{(\A,P)}$ and $\iota_{(\A',P')}$ be the corresponding maps for the pairs
$(\A,P)$ and $(\A',P')$ such that $\iota_{(\A,P)}(\P^2)=\iota_{(\A',P')}(\P^2)$. Let $T= \iota_{(\A',P')}^{-1} \circ \iota_{(\A,P)}: \P^2
\rightarrow \P^2$. Then, $T$ is an automorphism of $\P^2$ such that $T(L_i)=L'_i$ for every $i$ and $T(P)=P'$. Hence they are isomorphic, and so
we have the one-to-one correspondence.
\end{proof}

\begin{obs}
The following is a sketch of how this one-to-one correspondence works for arrangements of sections on geometrically ruled surfaces (cf. \cite[p.
369]{Hartshorne}) via the moduli spaces $\overline{M}_{0,d+1}$ (cf. \cite{KaVeronese93}). We will do it only for line arrangements, and over
$\C$. For the general case see \cite{Urzua4}.

Let us fix a pair $(\A,P)$ as before, and let $\text{Bl}_P(\P^2)$ be the blow-up of $\P^2$ at the point $P$ \cite[p. 386]{Hartshorne}. Then, we
have an induced genus zero fibration $\text{Bl}_P(\P^2) \rightarrow \P^1$. The pull-back of $\A$ in $\text{Bl}_P(\P^2)$ is an arrangement of $d$
labeled sections, each of which belongs to the fix class $E + F$. Here $E$ is the exceptional divisor of the blow-up, and $F$ is any fiber.
Conversely, given an arrangement of $d$ sections $\A$ in $\text{Bl}_P(\P^2)$ with members in the fix class $E+F$, we blow-down the exceptional
divisor $E$ to obtain a pair $(\A,P)$ in $\P^2$. Isomorphic pairs $(\A,P)$ correspond to isomorphic arrangements of sections (via automorphisms
of the fibration $\text{Bl}_P(\P^2) \rightarrow \P^1$).

Now the correspondence. We have fixed the pair $(\A,P)$, and the fibration $\text{Bl}_P(\P^2) \rightarrow \P^1$ as given above. Consider the
genus zero fibration $f: R \rightarrow \P^1$, where $R$ is the blow-up at all the singular points of $\A$ in $\text{Bl}_P(\P^2)$ except nodes.
Then, $f$ is a family of $(d+1)$-marked stable curves of genus zero. The markings are given by the labeled lines of $\A$, which are now $d$
labeled sections of $f$, and the $(-1)$-curve coming from the exceptional divisor $E$ in $\text{Bl}_P(\P^2)$. Therefore, since
$\overline{M}_{0,d+1}$ is a fine moduli space, we have the following commutative diagram coming from its universal family. $$ \xymatrix{ R
\ar[d]_{f} \ar[r] & \overline{M}_{0,d+2} \ar[d]_{\pi_{d+2}} \\ \P^1  \ar[r]^{g} & \overline{M}_{0,d+1}}$$

Let $B'$ be the image of $g$ in $\overline{M}_{0,d+1}$. It is a projective curve, since $f$ has singular and non-singular fibers, and so $f$ is
not isotrivial. Let us now consider the Kapranov map $\psi_{d+1}: \overline{M}_{0,d+1} \rightarrow \P^{d-2}$ \cite[p. 81]{KaChow93}, and let
$B=\psi_{d+1}(B')$. Because of the geometry of the fibers of $f$ and the Kapranov's construction, one can prove (see \cite{Urzua4}) that $B$
intersects all the hyperplanes $\Lambda_{i,j}$ transversally. Say $B$ intersects $\Lambda_{i,j}$. This means that the lines $L_i$ and $L_j$ of
$\A$ intersect in $\P^2$. But since they are lines, they can only intersect at one point. Therefore, we must have $B.\Lambda_{i,j}=1$, and so
$\deg(B)=1$, that is, $B$ is a line in $\P^{d-2}$. Observe that this line is outside of $\H_d$.

In particular, $B'$ is a smooth rational curve. It is not hard to see the converse, this is, how to obtain a pair $(\A,P)$ from a line in
$\P^{d-2}$ outside of $\H_d$ (see \cite{Urzua4}). Moreover, one can check that the pair we obtain is unique up to isomorphism of pairs. In this
way, to prove the one-to-one correspondence, we have to show that the map $g$ is an inclusion.

Assume $\deg(g)>1$. Notice that $g$ is totally ramified at the points corresponding to singular fibers of $f$, since again they come from
intersections of lines in $\P^2$, and so all the singular fibers have distinct points as images in $B'$. Let $\sing(f)$ be the set of points in
$\P^1$ corresponding to singular fibers of $f$. Then, since $\bigcap_{i=1}^d L_i=\emptyset$, we have $|\sing(f)|\geq 3$ (at least we have a
triangle in $\A$). Now, by the Riemann-Hurwitz formula, we have $$ -2=\deg(g) (-2) + (\deg(g)-1)|\sing(f)| + \epsilon $$ where $\epsilon \geq 0$
stands for the contribution from ramification of $f$ not in $\sing(f)$. But we re-write the equation as $0 = (\deg(g)-1)(|\sing(f)|-2) +
\epsilon$, and since $\deg(g)>1$ and $|\sing(f)|\geq 3$, this is a contradiction. Therefore, $\deg(g)=1$ and we have proved the one-to-one
correspondence. Again, we refer to \cite[Ch. 2 and 3]{Urzua4} for the general one-to-one correspondence involving arrangements of sections on
geometrically ruled surfaces. \label{r0}
\end{obs}

In this way, for each pair $(\A,P) \in \LL_d$, we denote its corresponding line in $\P^{d-2}$ by $L(\A,P)$. We now want to describe more
precisely how this one-to-one correspondence relates them.

\begin{defi}
Let $\K$ be any field. The pair $(\A,P)$ is said to be defined over $\K$ if the coefficients of the equations defining the lines in $\A$, and
the coordinates of $P$ are in $\K$. \label{d2}
\end{defi}

Hence, for arbitrary fields $\K$, Proposition \ref{p1} gives a one-to-one correspondence between pairs $(\A,P)$ defined over $\K$, and lines
$L(\A,P)$ in $\P^{d-2}$ defined over $\K$.

\begin{defi}
Let $1< k <d$ be an integer. A point in $\P^2$ is said to be a \underline{$k$-point} of $\A$ if it belongs to exactly $k$ lines of $\A$. If
these lines are $\{L_{i_1},L_{i_2},\ldots,L_{i_k} \}$, we denote this point by $[[i_1,i_2,\ldots,i_k]]$. The \underline{number of $k$-points} of
$\A$ is denoted by $t_k$. \label{d3}
\end{defi}

\begin{obs}
The complexity of an arrangement relies on its $k$-points. There are more constraints for the existence of an arrangement, over some field, than
the plane restriction: any two lines intersect at one point. Combinatorially there are possible line arrangements, with assigned $k$-points,
which may not be realizable in $\P^2$ over $\C$ (we will return to this in the next sections, for the particular case of nets). For instance, we
have the Fano arrangement (formed by seven lines with seven $3$-points) which is not realizable in $\P^2$ over fields of characteristic $\neq
2$. A rather trivial restriction, which is purely combinatorial, is that the numbers $t_k$ must satisfy  $ d \choose 2$ $= \sum_{k=2}^d$ $k
\choose 2$ $ t_k$; this is the only linear relation they satisfy for a fix $d$. In \cite{HiLines83}, Hirzebruch proved the following inequality
for an arrangement of $d$ lines in the complex projective plane having $t_d=t_{d-1}=0$, $$ t_2 + \frac{3}{4} t_3 \geq d + \sum_{k\geq 5} (k-4)
t_k.$$ This is a non-trivial relation among the numbers $t_k$, which comes from the Miyaoka-Yau inequality for complex algebraic surfaces (see
\cite{Urzua4} for more about this type of restrictions). This inequality is clearly not true in positive characteristic. \label{r1}
\end{obs}

Let us fix a pair $(\A,P)$, and its line $L(\A,P)$ in $\P^{d-2}$. Let $\lambda$ be a line in $\P^2$ passing through $P$. Notice that $\lambda$
corresponds to a point in $L(\A,P)$. Let $K(\lambda)$ be the set of $k$-points of $\A$ in $\lambda$, for all $1<k<d$; it might be empty or
consist of several points. We write $$ K(\lambda)= \{ [[i_1,i_2,\ldots,i_{k_1}]], [[j_1,j_2,\ldots,j_{k_2}]],\ldots \}.$$

\begin{example}
In Figure \ref{f0}, we have the complete quadrilateral $\A$, formed by the set of lines $\{ L_1, \ldots , L_6 \}$, and a point $P$ outside of
$\A$. Through $P$ we have all the $\lambda$ lines. In the figure, we have named two such lines: $\lambda$ and $\lambda '$. Thus, $K(\lambda)=\{
[[3,6]], [[1,4]] \}$ and $K(\lambda ')= \{ [[1,2,3]] \}$.

\begin{figure}[htbp]
\includegraphics[width=12cm]{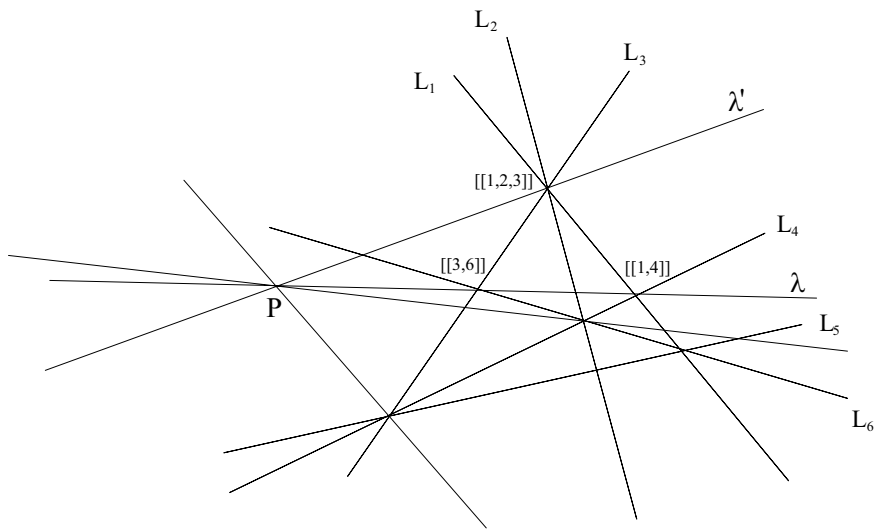} \caption{Some $K(\lambda)$ sets for the pair $($Complete quadrilateral$,P)$.} \label{f0}
\end{figure}

\label{e}
\end{example}

The set $K(\lambda)$ imposes the following constraints for the the point $[x_1:x_2:\ldots:x_{d-1}]$ in $L(\A,P)$ corresponding to $\lambda$. For
each $k$-point $[[i_1,i_2,\ldots,i_{k}]]$ in $K(\lambda)$, we have:
\begin{itemize}
\item If for some $j$, $i_j=d$, then $x_{i_l}=0$ for all $i_l \neq d$.

\item Otherwise, $x_{i_1}=x_{i_2}=\ldots=x_{i_k}\neq 0$.
\end{itemize}
For $[[i_1,\ldots,i_{k_1}]]$, $[[j_1,\ldots,j_{k_2}]]$ in $K(\lambda)$, we have that $x_{i_a} \neq x_{j_b}$, otherwise $[[i_1,\ldots,i_{k_1}]]$
and $[[j_1,\ldots,j_{k_2}]]$ would not be distinct points in $\lambda$. We will work out various examples when we compute nets in the next
sections.

Assume we know the combinatorial data of $(\A,P)$, but we do not know whether is realizable in $\P^2$ over some field $\K$. Then, this
realization question is equivalent to the realization question of $L(\A,P)$ over $\K$. If we are only interested in the line arrangement $\A$,
the point $P$ introduces unnecessary dimensions which makes the realization question harder. Instead, we consider the new pair $(\A',P')$, where
$P' \in \A$ and the lines of $\A'$ are the lines in $\A$ not containing $P'$ in a certain order. Now, the line $L(\A',P')$ corresponding to
$(\A',P')$ is in $\P^{d'-2}$, and $d'<d$. So we have less dimensions to work with, and $L(\A',P')$ completely represents our arrangement $\A$,
by keeping track of $P'$.

If we take $P'$ as a $k$-point with $k$ large, the previous observation will be important to simplify computations to prove or disprove the
realization of $\A$. In addition, we find a moduli space for the combinatorial type of $\A$, forgetting the artificial point $P$. Again, by
combinatorial type we mean the data given by some of the intersection of its lines.

In the next two sections, we will compute some special configurations by means of the line $L(\A,P)$. We make the following choices to write
down equations for the lines in $\A$:

\begin{itemize}
\item The point $P$ will be always $[0:0:1]$. \item The arrangement $\A$ will be formed by $\{ L_1,\ldots,L_d \}$, where $L_i$ are the lines of
$\A$, and also their linear polynomials $L_i(x,y,z)=(a_ix+b_iy+z)$ for every $i\neq d$, and $L_d=(z)$.
\end{itemize}

With these assumptions, it is easy to check that the corresponding line $L(\A,P)$ in $\P^{d-2}$ is $[a_it+b_iu]_{i=1}^{d-1}$, where $[t:u] \in
\P^1$.

%--------------------------------------------------------------------------------------------------------------------------------------------
\section{$(p,q)$-nets in $\P^2$.}\label{s3}

We now introduce a specific type of line arrangements in $\P^2$ which are called nets. Our main references are \cite{DeKe74}, \cite{LibYuz00},
\cite{YuzNets04}, \cite{kawa07}, \cite{Stipins1}, and \cite{Stipins2}. We begin with the definition of a net taken from \cite{Stipins1}.

\begin{defi}
Let $p \geq 3$ be an integer. A \underline{$p$-net} in $\P^2$ is a $(p+1)$-tuple $(\A_1,...,\A_p,\X)$, where each $\A_i$ is a nonempty finite
set of lines of $\P^2$ and $\X$ is a finite set of points of $\P^2$, satisfying the following conditions:
\begin{itemize}
\item[(1)] The $\A_i$ are pairwise disjoint. \item[(2)] The intersection point of any line in $\A_i$ with any line in $\A_j$ belongs to $\X$ for
$i\neq j$. \item[(3)] Through every point in $\X$ there passes exactly one line of each $\A_i$.
\end{itemize}
\label{d4}
\end{defi}

One can prove that $|\A_i|=|\A_j|$ for every $i,j$ and $|\X|=|\A_1|^2$ (see \cite{Stipins1}, \cite{YuzNets04}). Let us denote $|\A_j|$ by $q$,
this is the \underline{degree} of the net. Thus, if we use classical notation (see for example \cite{Do04} or \cite{Gr06}), a $p$-net of degree
$q$ is a $(q^2_p,pq_q)$ configuration. Following \cite{Stipins1} and \cite{YuzNets04}, we denote a $p$-net of degree $q$ by
\underline{$(p,q)$-net}. We label the lines of $\A_i$ by $\{ L_{q(i-1)+j} \}_{j=1}^q$ for all $i$, and define the arrangement
$\A=\{L_1,L_2,...,L_{pq} \}$. We assume $q\geq 2$ to get rid of the trivial arrangement, which is actually not considered in Definition
\ref{d1}.

Assume for now that we work over an algebraically closed field $\K$. A $(p,q)$-net $\A = (\A_1,...,\A_p,\X)$ defines a unique pencil of curves
$\PP(\A)$ of degree $q$ as follows. Take any two sets of lines $\A_i$ and $\A_j$. Consider $\A_i$ and $\A_j$ as the equations which define them,
i.e., the multiplication of its lines. Then, the pencil is defined as $$\PP(\A)= \{ u \A_i + t \A_j \ : \ [u:t] \in \P^1 \}.$$ This is
well-defined. Take $\A_k$ with $k \neq i,j$, and a point $Q$ in $\A_k \setminus \X$. Then, there exists $[u:t] \in \P^1$ such that $u \A_i(Q) +
t \A_j(Q)=0$. We write $\B=u \A_i + t \A_j$, which is a curve of degree $q$ containing $\X \cup \{Q \}$. If $\A_k$ and $\B$ do not have common
factors, we have, by Bezout's Theorem, that $\A_k$ belongs to $\PP(\A)$ ($\A_k$ is $\B$ times a non-zero constant). This proves the independence
of the choice of $i,j$ to define $\PP(\A)$. If $\A_k$ and $\B$ have a non-trivial common factor, then it has to be formed by the multiplication
of $0<q_1<q$ lines in $\A_k$. In this way, this common factor $C$ contains exactly $qq_1$ points of $\X$. Therefore, if $\B=CF$ and $\A_k=CG$,
the set $\{F=0 \} \cap \{G=0 \}$ has at least $q(q-q_1)>(q-q_1)^2$ points, $\deg(F)=\deg(G)=q-q_1$, and $gcd(F,G)=1$. This is impossible by
Bezout's Theorem.

In addition, if the characteristic of $\K$ is zero, the general member of this pencil is smooth \cite[p. 272]{Hartshorne}, i.e., outside of
finitely many points in $\P^1$, $uA_i+tA_j$ is a smooth plane curve. Hence, after we blow up the $q^2$ points in $\X$ we obtain a fibration of
curves of genus $\frac{(q-1)(q-2)}{2}$ with at least $p$ completely reducible fibers. This fibration leads to the following restriction on nets
defined over $\C$, due to Yuzvinsky \cite{YuzNets04} (see \cite{PeYuz07} for the higher dimensional analogue). The proof is a simple topological
argument which uses the topological Euler characteristic of the fibration.

\begin{prop}
For an arbitrary $(p,q)$-net in $\P^2$ defined over $\C$, the only possible values for $(p,q)$ are: $(p=3,q\geq 2)$, $(p=4,q\geq 3)$ and $(p=5,q
\geq 6)$. \label{p2}
\end{prop}

The combinatorial data which defines $(p,q)$-nets can be expressed using Latin squares. A \underline{Latin square} is a $q\times q$ table filled
with $q$ different symbols (in our case numbers from $1$ to $q$) in such a way that each symbol occurs exactly once in each row and exactly once
in each column. They are the multiplication tables of finite quasigroups. Let $\A = (\A_1,...,\A_p,\X)$ be a $(p,q)$-net. The $q^2$ $p$-points
in $\X$ are determined by $(p-2)$ $q\times q$ Latin squares which form an orthogonal set, as explained for example in \cite{Stipins1} or
\cite{kawa07}.

Although we have defined nets as arrangements of lines already in $\P^2$, we will first ``think combinatorially" about the $(p,q)$-net through
this orthogonal set of $(p-2)$ Latin squares, and then we will attempt to prove or disprove its realization on $\P^2$ over some field. This is
the strategy from now on.

\begin{example}
In this example we use our correspondence to reprove the existence of the famous \underline{Hesse arrangement}. This $(4,3)$-net has nice
applications in algebraic geometry (see for example \cite{HiLines83,ArDo06,Urzua4}). Let us denote this net by $\A = \A_1 \cup \A_2 \cup \A_3
\cup \A_4$, with $\A_i = \{L_{3i-2}, L_{3i-1},L_{3i} \}$. By relabelling the lines of $\A$, we may assume that the combinatorial data is given
by the following set of orthogonal Latin squares.

\begin{center}
\vspace{0.3 cm}{\tiny
\begin{tabular}{|rcccc}
\hline
   1 & 2 & 3     \\
   2 & 3 & 1     \\
   3 & 1 & 2     \\

\end{tabular}
\hspace{1.5 cm}
\begin{tabular}{|rcccc}
\hline
   1 & 2 & 3     \\
   3 & 1 & 2     \\
   2 & 3 & 1     \\

\end{tabular}}
\vspace{0.3 cm}
\end{center}

These Latin squares give the intersections of $\A_3$ and $\A_4$ respectively with $\A_1$ (columns) and $\A_2$ (rows) (see \cite{kawa07} or
\cite{Stipins1}). For example, the left one tell us that $L_2$, $L_6$ and $L_7$ (values) have a common point of incidence. The right one says
$L_2$, $L_6$ and $L_{12}$ have also non-empty intersection. Hence, $[[2,6,7,12]] \in \X$. In this way, we find $\X$, which is completely
described in the following tables.

\begin{center} \vspace{0.3 cm}
\begin{tabular}{c|ccc}

         & $L_1$ & $L_2$ & $L_3$ \\ \hline
   $L_4$ & $L_7$ & $L_8$ & $L_9$ \\
   $L_5$ & $L_8$ & $L_9$ & $L_7$ \\
   $L_6$ & $L_9$ & $L_7$ & $L_8$ \\

\end{tabular}
\hspace{1.5 cm}
\begin{tabular}{c|ccc}

         & $L_1$ & $L_2$ & $L_3$ \\ \hline
   $L_4$ & $L_{10}$ & $L_{11}$ & $L_{12}$ \\
   $L_5$ & $L_{12}$ & $L_{10}$ & $L_{11}$ \\
   $L_6$ & $L_{11}$ & $L_{12}$ & $L_{10}$ \\

\end{tabular}
\vspace{0.3 cm}
\end{center}

We now consider the new arrangement of lines $\A'=\A\setminus \{L_3,L_4,L_9,L_{12} \}$ together with the point $P=[[3,4,9,12]]$. We rename the
twelve lines in the following way: $\A'= \{L'_1=L_1, L'_2=L_2, L'_3=L_5, L'_4=L_6, L'_5=L_7, L'_6=L_8, L'_7=L_{10}, L'_8=L_{11} \}$ and the
lines passing through $P$, $\alpha=L_3$, $\beta=L_4$, $\gamma=L_9$, and $\delta=L_{12}$. By our correspondence, we have a line $L(\A',P)$ in
$\P^6$ for the pair $(\A',P)$, and it passes through these distinguished four points $\alpha$, $\beta$, $\gamma$, and $\delta$ (we abuse the
notation, these lines correspond to points in $L(\A',P)$). Then, $K(\alpha)=\{ [[4,6,7]], [[3,5,8]]\}$, $K(\beta)=\{[[2,6,8]],[[1,5,7]] \}$,
$K(\gamma)$ $=\{[[2,3,7]],[[1,4,8]] \}$, and $K(\delta)= \{[[2,4,5]], [[1,3,6]] \}$. Hence, we write:$$ \alpha=[a_1:a_2:0:1:0:1:1], \ \ \ \
\beta=[1:0:a_3:a_4:1:0:1]$$ $$\gamma=[0:1:1:0:a_5:a_6:1], \ \ \ \ \delta=[1:a_7:1:a_7:a_7:1:a_8]
$$ for some numbers $a_i$ (with extra restrictions), and we take $L(\A',P): \alpha t + \beta u$, $[t:u] \in \P^1$.

For some $[t:u]$, we have the equation $ \alpha t + \beta u = \gamma$, and from this we obtain:
$$ a_1=\frac{w-1}{w} \ \ a_2= \frac{1}{w} \ \ a_3= \frac{1}{1-w} \ \ a_4= \frac{w}{w-1} \ \ a_5= 1-w \ \ a_6=w,$$ where $w$ is a parameter.

For another pair $[t:u]$, we have $ \alpha t + \beta u = \delta$, and so $w^2-w+1=0$, $a_7=\frac{1}{w}$, and $a_8=\frac{1+w}{w}$. Therefore, our
field of definition needs to have roots for the equation $w^2-w+1$. For instance, over $\C$, we take $w=e^{\frac{\pi \sqrt{-1}}{3}}$, and then
$L(\A',P)$ is: $$[w-1:1:0:w:0:w:w]t+[w-1:0:-1:w:w-1:0:w-1]u.$$ According to our choices at the end of Section \ref{s2}, we write down the lines
of $\A$ as: $$\{ L_1=((w-1)x+ y + z), L_2=(x + z), L_3=(y) \} \ \ \{ L_4=(x), L_5=(-y + z), L_6=(wx + w y + z) \}$$
$$ \{ L_7=(y+w^2z), L_8=(wx+z), L_9=(wx+y) \} \ \ \{ L_{10}=(x+wy+\frac{1}{w}z), L_{11}=(z), L_{12}=(x-wy) \}.$$ Notice that the lines in $\A$
corresponding to $\alpha$, $\beta$, $\gamma$, and $\delta$ are $ux-ty=0$, where $[t:u]$ is the corresponding point in $\P^1$ for each of them,
as points in $L(\A',P)$. \label{e1}
\end{example}

\begin{example}
In this example, we show that there are no $(4,4)$-nets over fields of characteristic $\neq 2$. This fact has independently been shown in
\cite{DMWZ07} over $\C$. We start supposing their existence, let $\A=\{\A_i \}_{i=1}^4$ be such a net. Again, by relabelling the lines of $\A$,
we may assume that the orthogonal set of Latin squares is:
\begin{center}
\vspace{0.3 cm}{\tiny
\begin{tabular}{|rcccc}
\hline
   1 & 2 & 3 & 4  \\
   2 & 1 & 4 & 3  \\
   3 & 4 & 1 & 2  \\
   4 & 3 & 2 & 1
\end{tabular}
\hspace{1.5 cm}
\begin{tabular}{|rcccc}
\hline
   1 & 2 & 3 & 4  \\
   3 & 4 & 1 & 2  \\
   4 & 3 & 2 & 1  \\
   2 & 1 & 4 & 3
\end{tabular}}
\vspace{0.3 cm}
\end{center}

We consider $(\A',P)$ defined by the arrangement of twelve lines $\A'=\A\setminus \{L_4,L_5,L_{12},L_{16} \}$, and the point $P=[[4,5,12,16]]$.
The lines of $\A'$ are $L'_1=L_1$, $L'_2=L_2$, $L'_3=L_3$, $L'_4=L_6$, $L'_5=L_7$, $L'_6=L_8$, $L'_7=L_9$, $L'_8=L_{10}$, $L'_9=L_{11}$,
$L'_{10}=L_{13}$, $L'_{11}=L_{14}$, and $L'_{12}=L_{15}$. The special lines are $\alpha=L_4$, $\beta=L_5$, $\gamma=L_{12}$, and $\delta=L_{16}$.
Hence, we have that $$\alpha=[a_1:a_2:a_3:1:a_4:0:0:a_4:1:a_4:1] \ \ \ \beta=[1:b_4:0:b_1:b_2:b_3:1:b_4:0:1:b_4]$$
$$\gamma=[1:0:c_4:c_4:0:1:c_1:c_2:c_3:c_4:1] \ \ \ \delta=[d_1:1:d_4:1:d_1:d_4,1:d_4:d_1:d_2:d_3]$$ as points in $L(\A',P)$, which we write as
$\alpha t+ \beta u$, $[t:u] \in \P^1$. Let $c_1=a$, $c_2=b$ and $c_3=c$. Since $\gamma \in L(\A',P)$, we have:
$$ a_1= \frac{1-a}{c} \ \ a_2= \frac{c-1}{c} \ \ a_3= \frac{a+b+c-1}{c} \ \ a_4= \frac{b+c-1}{c} \ \ c_4=a+b+c-1$$
$$ b_1= \frac{a+b-1}{a} \ \ b_2= \frac{1-b-c}{a} \ \ b_3= \frac{1}{a} \ \ b_4= \frac{1-c}{a}.$$ Also, since $\delta \in L(\A',P)$, have $ad_4=1$
and $ad_1+b=1$, plus the following equations: $(1): d_1(1-c)=1-b-c$, $(2): d_1(1-b)(c-1)+d_1c(1-c)=(1-b)c$, $(3): (1-b)(1+d_4(b+c-1))=d_4c$, and
$(4):c^2=(1-b)(b+c-1)$ among others. These equations are enough to produce a contradiction. By isolating $d_1$ in $(1)$, replacing it in $(2)$,
and using $(4)$, we get $c^3=(1-b)^3$ which requires a $3$rd primitive root of $1$. Say $w$ is such, so $b=1-w c$. Then, by $(3)$, we get
$w^2(1+2c)=w-1$. Since the characteristic of our field is not $2$, we have $c= \frac{1}{w}$, and so $b=0$, and $a=1$. This gives $a_1=0$, which
is a contradiction, because it would imply that $L_1 \cap L_4 \cap L_8 \cap L_9 \cap L_{15} \neq \emptyset $. See next example for the char. $2$
case.

\label{e2}
\end{example}

\begin{example}
Positive characteristic gives more freedom for the realization of nets compared to Proposition \ref{p2}, and the previous examples. Let $q$ be a
prime number, and let $\K$ be a field with $m=q^n$ elements. In $\P^2_{\K}$, we have $m^2+m+1$ points with coordinates in $\K$, and there are
$m^2+m+1$ lines such that through each of these points passes exactly $m+1$ of these lines, and each of these lines contains exactly $m+1$ of
these points \cite[p. 65]{Hirsch79}. By eliminating one of these lines, we obtain a $(m+1,m)$-net. Each of the $m+1$ members of this net has $m$
lines intersecting at one point, and so $t_m=m+1$, $t_{m+1}=m^2$, and $t_k=0$ otherwise. Hence, in positive characteristic, there are $p$-nets
for all $p\geq 3$. If we want a $(4,4)$-net, one takes $q=2$ (necessary by Example \ref{e2}) and $n=2$, and considers the corresponding
$(5,4)$-net. We now eliminate one of its members to obtain a $(4,4)$-net. \label{e3}
\end{example}

In \cite{Stipins2}, Stipins proves that there are no $5$-nets over $\C$ (see \cite{Yuz08} for a generalization of his result). His proof does
not use the combinatorics given by Latin squares. We will see that this issue matters for the realization of $(3,6)$-nets, and so, it would be
interesting to know if Latin squares are relevant or not for the possible realization of $4$-nets over $\C$. It is believed that, except for the
Hesse arrangement, $(4,q)$-nets do not exist over $\C$. In this way, by Proposition \ref{p2}, the only cases left over $\C$ would be $3$-nets.
In \cite{YuzNets04}, it is proved that for every finite subgroup $H$ of a smooth elliptic curve, there exists a $3$-net over $\C$ corresponding
to the Latin square of the multiplication table of $H$. In the same paper, the author proves that there are no $(3,8)$-nets associated to the
group $\Z/2\Z \oplus \Z/2\Z \oplus \Z/2\Z$. In \cite{Stipins1}, it can be found a classification of $(3,q)$-nets for $q \leq 5$. In the next
section we classify $(3,q)$-nets for $q \leq 6$, and the $(3,8)$-nets associated to the Quaternion group.

\begin{obs} (\textbf{Main classes of Latin squares})
As we explained before, a $q\times q$ Latin square gives the set $\X$ for a $(3,q)$-net $\A=\{ \A_1,\A_2,\A_3 \}$. What if we are interested
only in the realization of $\A$ in $\P^2$ as a curve, i.e., without labelling lines? Then, we divide the set of all $q \times q$ Latin squares
into the so-called \underline{main classes} (see \cite{DeKe74} or \cite{kawa07}).

For a given $q\times q$ Latin square $M$ corresponding to $\A$, by rearranging rows, columns and symbols of $M$, we obtain a new labelling for
the lines in each $\A_i$. If we write $M$ in its orthogonal array representation, i.e. $M=\{(r,c,s): r=\text{row number}, \ c=\text{column
number}, \ s=\text{symbol number} \}$, we can perform six operations on $M$, each of them a permutation of $(r,c,s)$ which translates into
relabelling the members $\{\A_1, \A_2, \A_3 \}$, and so we obtain the same curve in $\P^2$. We can partition the set of all $q\times q$ Latin
squares in main classes (also called Species) which means: if $M,N$ belong to the same class, then we can obtain $N$ by applying a finite number
of the above operations to $M$. In what follows, we will choose one member from each class. The following table shows the number of main classes
for small $q$.

\vspace{0.1cm}
\label{r2}
\end{obs}
\begin{center}
\begin{tabular}{|c|c|c|c|c|c|c|c|c|c|c|c|c|c|c|c|c|c|c}
\hline $q$ &  1  &  2  &  3  &  4  &  5  &  6  &  7  &  8  & 9 & 10 \\ \hline
$\#$ main classes &  1 & 1 & 1 & 2 & 2 & 12 & 147 & 283\,657 & 19\,270\,853\,541 & 34\,817\,397\,894\,749\,939\\
\hline
\end{tabular}
\end{center}
%\vspace{0.2cm}

%------------------------------------------------------------------------------------------------------------------------------------------------
\section{Classification of $(3,q)$-nets for $q\leq6$, and the Quaternion nets.}\label{s4}

In order to do this classification, we use again the trick of eliminating some lines passing through a $k$-point $P$, and considering the new
pair $(\A',P)$. We work with $(3,q)$-nets, thus $P$ is taken as a $3$-point in $\X$ (and so, we eliminate three lines from $\A$). If the
$(3,q)$-net is given by $\A=\{\A_1,\A_2,\A_3 \}$ such that $\A_i=\{L_{q(i-1)+j} \}_{j=1}^q$, then the new pair $(\A',P)$ will be given by $\A'=
\{L'_1=L_2,L'_2=L_3,\ldots,L'_{q-1}=L_q,L'_{q}=L_{q+2},L'_{q+1}=L_{q+3},\ldots,L'_{2q-2}=L_{2q},L'_{2q-1}=L_{2q+2},L'_{2q}=L_{2q+3},\ldots,L'_{3q-3}=L_{3q}
\}$, $P=L_1 \cap L_{q+1} \cap L_{2q+1}$, and $\alpha=L_1$, $\beta=L_{q+1}$, $\gamma=L_{2q+1}$. The corresponding line $L(\A',P)$ is $\alpha t +
\beta u$, $[t,u] \in \P^1$.

We obtain $\X$ from a given Latin square. Then, we fix a point $P$ in $\X$, so the locus of the line $L(\A',P)$ is actually the moduli space of
the $(3,q)$-nets with combinatorial data defined by that Latin square (or better its main class). We give in each case equations for the lines
of the nets depending on parameters coming from $L(\A',P)$.

\vspace{0.3 cm} \underline{\textbf{$(3,2)$-nets.}}

Here we have one main class given by the multiplication table of $\Z/2\Z$: {\tiny \begin{tabular}{|cc} \hline   1 & 2 \\    2 & 1 \\
\end{tabular}}. According to our set up, $(\A',P)$ is formed by an arrangement $\A'$ of three lines and $P=[[1,3,5]] \in \X$. The line
$L(\A',P)$ is actually the whole $\P^1$. This tells us that there is only one $(3,2)$-net, up to projective equivalence. The special points are
$\alpha=[1:0]$, $\beta=[0:1]$, and $\gamma=[1:1]$. This $(3,2)$-net is represented by the singular members of the pencil $\lambda z(x-y)+ \mu
y(z-x)=0$ on $\P^2$, and it is called complete quadrilateral (see Figure \ref{f0}).

\vspace{0.3 cm} \underline{\textbf{$(3,3)$-nets.}}

Again, there is one main class given by the multiplication table of $\Z/3 \Z$. $${\tiny \begin{tabular}{|ccc} \hline
1 & 2 & 3 \\
3 & 1 & 2 \\
2 & 3 & 1
\end{tabular}}$$

For $(\A',P)$ we have an arrangement of six lines $\A'$ and $P=[[1,4,7]] \in \X$, the line $L(\A',P)$ is in $\P^4$. The special points can be
taken as $\alpha= [a_1:a_2:1:0:1]$, $\beta=[1:0:b_1:b_2:1]$, and $\gamma=[1:c_1:c_1:1:c_2]$. Then, for some $[t:u] \in \P^1$, we have $\alpha
t+\beta u=\gamma$. Thus, if $a_2=a$, $b_2=b$ and $c_1=c$, we have that $\alpha=\big[ \frac{a(b-1)}{bc}:a:1:0:1 \big]$ and
$\beta=\big[1:0:\frac{bc(a-1)}{a}:b:1 \big]$. The rest of the points in $\X'$ (again, although $\A'$ is not a net, we think of $\X'$ as the set
of $3$-points in $\A'$ coming from $\X$) $[[1,3,6]]$ and $[[2,4,5]]$ give the same restriction $(a-1)(b-1)=1$, i.e., $a= \frac{b}{b-1}$.
Therefore, the line $L(\A',P)$ has two parameters of freedom, and it is given by $ \bigl[\frac{1}{c}:\frac{b}{b-1}:1:0:1 \bigr] t +
[1:0:c:b:1]u$ where $c,b$ are numbers with some restrictions (for example, $c,b\neq 0$ or $1$). Hence, we find that this family of $(3,3)$-nets
can be represented by: $L_1=(y)$, $L_2=(\frac{1}{c}x+y+z)$, $L_3=(\frac{b}{b-1}x+z)$, $L_4=(x)$, $L_5=(x+cy+z)$, $L_6=(by+z)$,
$L_7=(x+c(1-b)y)$, $L_8=(x+y+z)$, and $L_9=(z)$.

\vspace{0.3 cm} \underline{\textbf{$(3,4)$-nets.}}

Here we have two main classes. We represent them by the following Latin squares.

\vspace{0.1 cm}
\begin{center}
{\tiny $M_1=$
\begin{tabular}{|cccc}
\hline
   1 & 2 & 3 & 4  \\
   2 & 3 & 4 & 1  \\
   3 & 4 & 1 & 2  \\
   4 & 1 & 2 & 3
\end{tabular}
\hspace{1.5 cm} $M_2=$
\begin{tabular}{|cccc}
\hline
   1 & 2 & 3 & 4  \\
   2 & 1 & 4 & 3  \\
   3 & 4 & 1 & 2  \\
   4 & 3 & 2 & 1
\end{tabular}}
\vspace{0.2 cm}
\end{center}

They correspond to $\Z/4 \Z$ and $\Z/2\Z \oplus \Z/2\Z$ respectively. We first deal with $M_1$. Then, we have
$\alpha=[a_1:a_2:a_3:1:a_4:0:1:a_4]$, $\beta=[1:b_1:0:b_2:b_3:b_4:1:b_1]$ and $\gamma=[1:c_1:c_2:c_2:c_1:1:c_3:c_4]$. Let $a_3=a$, $b_4=b$ and
$c_2=c$. By imposing $\gamma$ to $L(\A',P)$, one can find $a_1= \frac{(-1+b)a}{bc}$, $a_2=\frac{(-b_1+c_1 b)a}{bc}$, $a_4=\frac{(-b_1 + c_4
b)a}{bc}$, $b_2=\frac{-(c-c_2a)b}{a}$, $b_3=b_1-c_4b+c_1b$, and $c_3=\frac{1}{b} + \frac{c}{a}$. When we impose $L(\A',P)$ to pass through
$[[1,5,9]]$, $[[2,4,9]]$ and $[[1,4,8]]$, we obtain equations to solve for $c_4$, $c_1$, and $b_1$ respectively. After that, the restrictions
$[[2,6,7]]$, $[[3,5,7]]$, and $[[3,6,8]]$ are trivially satisfied. The line $L(\A',P)$ is parametrized by $(a,b,c)$ in a open set of $\Af^3$,
and it is given by: $a_1= \frac{a(b-1)}{bc}$, $a_2= \frac{ab}{abc+ab-a-bc}$, $a_3= a$, $a_4= \frac{a^2(b-1)}{abc+ab-a-bc}$,
$b_1=\frac{b^2(a-1)c}{abc+ab-a-bc}$, $b_2=\frac{bc(a-1)}{a}$, $b_3=\frac{abc}{abc+ab-a-bc}$ and $b_4=b$.

Similarly, for $M_2$ we have $\alpha=[a_1:a_2:a_3:1:a_4:0:1:a_4]$, $\beta=[1:b_1:0:b_2:b_3:b_4:1:b_1]$, and
$\gamma=[1:c_1:c_2:1:c_1:c_2:c_3:c_4]$. Of course, the only change with respect to the previous case is $\gamma$. By doing similar computations,
we have that $L(\A',P)$ is parametrized by $(a,b,c)$ in a open set of $\Af^3$, and it is given by: $a_1= \frac{(b-c)a}{bc}$, $a_2=
\frac{abc}{abc+ab-bc-ac}$, $a_3= a$, $a_4= \frac{a^2(b-c)}{abc+ab-bc-ac}$, $b_1=\frac{b^2(a-c)}{abc+ab-bc-ac}$, $b_2=\frac{b(a-c)}{ac}$,
$b_3=\frac{abc}{abc+ab-bc-ac}$, and $b_4=b$ (see \cite[p. 11]{Stipins1} for more information about this net).

Hence, the lines for the corresponding $(3,4)$-nets for $M_r$ can be represented by: $L_1=(y)$, $L_2=(a_1 x+y+z)$, $L_3=(a_2 x+ b_1 y + z)$,
$L_4=(a_3 x + z)$, $L_5=(x)$, $L_6=(x+b_2 y+z)$, $L_7=(a_4 x+b_3 y+z)$, $L_8=(b_4 y+z)$, $L_9=(a x - b c^{2-r}y)$, $L_{10}=(x+y+z)$,
$L_{11}=(a_4x + b_1y +z)$, and $L_{12}=(z)$. For example, if we evaluate the equations for the cyclic type $M_1$ at $a= \frac{1+i}{2}$,
$b=\frac{1-i}{2}$, and $c=-i$ (where $i=\sqrt{-1}$), we obtain the well-known net: $\A_1= \{ y, (1+i)x+2y+2z, (1+i)x+y+2z, (1+i)x+2z \}$,
$\A_2=\{ x, 2x+(1-i)y+2z, x+(1-i)y+2z, (1-i)y+2z \}$ and $\A_3= \{ x+y, x+y+z, x+y+2z, z \}$. This net is projectively equivalent of the one
given by the plane curve $(x^4-y^4)(y^4-z^4)(x^4-z^4)=0$, known as CEVA$(4)$ \cite[p. 435]{Do04}.

\vspace{0.3 cm} \underline{\textbf{$(3,5)$-nets.}}

We have two main classes, and we represent them by the following Latin squares.
\begin{center}
\vspace{0.1 cm}{\tiny $M_1=$
\begin{tabular}{|ccccc}
\hline
   1 & 2 & 3 & 4 & 5 \\
   2 & 3 & 4 & 5 & 1 \\
   3 & 4 & 5 & 1 & 2 \\
   4 & 5 & 1 & 2 & 3 \\
   5 & 1 & 2 & 3 & 4
\end{tabular}
\hspace{1.5 cm} $M_2=$
\begin{tabular}{|ccccc}
\hline
   1 & 2 & 3 & 4 & 5 \\
   2 & 1 & 4 & 5 & 3 \\
   3 & 5 & 1 & 2 & 4 \\
   4 & 3 & 5 & 1 & 2 \\
   5 & 4 & 2 & 3 & 1
\end{tabular}}
\vspace{0.1 cm}
\end{center}

The Latin square $M_1$ corresponds to $\Z/5\Z$. As before, for $M_1$ and $M_2$ we have that $\alpha= [a_1:a_2:a_3:a_4:1:a_5:a_6:0:1:a_5:a_6]$
and $\beta=[1:b_1:b_2:0:b_3:b_4:b_5:b_6:1:b_1:b_2]$, but for $M_1$, $\gamma=[1:c_1:c_2:c_3:c_3:c_2:c_1:1:c_4:c_5:c_6]$, and for $M_2$,
$\gamma=[1:c_1:c_2:c_3:1:c_1:c_2:c_3:c_4:c_5:c_6]$.

In the case of $M_1$, after we impose $\gamma$ to $L(\A',P)$, we use the conditions $[[2,3,5]]$, $[[4,8,11]]$, $[[2,6,12]]$, $[[2,8,9]]$, and
$[[3,8,10]]$ to solve for $b_2$, $c_6$, $c_5$, $b_1$, and $c_2$ respectively. After that we have four parameters left: $a_4=a$, $b_6=b$,
$c_3=c$, and $c_1=d$, and we get the following constrain for them:$$b^2(a-1)(d-c)(c-ad)+ b(-d^2a+dc+2d^2a^2-2da^2c-da+ca-dc^2+c^2da) +
ad(ca-da+1-c)= 0.$$

Hence, the $(3,5)$-nets for $M_1$ are parametrized by an open set of the hypersurface in $\Af^4$ defined by this equation. The values for the
variables are:
$$ a_1=\frac{a(b-1)}{bc} \ \ a_2=\frac{ab(d-1)}{a-ba+bc} \ \ a_3=\frac{a(d-db+bc)}{c^2(a-1)b} \ \ a_4=a $$ $$a_5=\frac{a^2(d-1)(d-db+bc)}{(a-ba+bc)
(a-1)cd} \ \ a_6= \frac{ad(b-1)}{bc}\ \ b_1=\frac{b(da-adb+bc)}{a-ba+bc} \ \ b_2=\frac{d-db+bc}{c}$$ $$b_3=\frac{bc(a-1)}{a} \ \ b_4=
\frac{(da-adb+bc)(d-db+bc)a}{(a-ba+bc)(a-1)cd} \ \ b_5=d \ \ b_6= b. $$

In the case of $M_2$, we obtain a three dimensional moduli space of $(3,5)$-nets as well. It is parametrized by $(a,b,c)$ in an open set of
$\Af^3$ such that $a_4=a$, $b_6=b$ and $c_1=c$, and:
$$ a_1=\frac{a^2(1-b)}{b(ab-a-b)} \ \ a_2=c \ \ a_3=\frac{(-a^2+a^2b+cba-ab-cb)b}{(ab+cb-a-b)(ab-a-b)} \ \ a_4=a $$
$$a_5=\frac{(a^2-a^2b-cba+ab+cb)a}{(ab-a-b)^2} \ \ a_6= \frac{c(b-1)a}{-a+ab+cb-b}$$ $$b_1=\frac{cb^2(1-a)}{a(ab-a-b)} \ \ b_2=\frac{(a-ab+b-c)b^2}{(-a+ab+cb-b)(ab-a-b)}$$
$$b_3=\frac{b^2(1-a)}{a(ab-a-b)} \ \ b_4= \frac{ab(ab-a-b+c)}{(ab-a-b)^2} \ \ b_5=\frac{cb(a+b-ab)}{a(ab-a+bc-b)} \ \ b_6= b. $$

To obtain the lines for the nets corresponding to $M_r$, we just evaluate: $L_1=(y)$, $L_2=(a_1x+y+z)$, $L_3=(a_2x+b_1y+z)$,
$L_4=(a_3x+b_2y+z)$, $L_5=(a_4x+z)$, $L_6=(x)$, $L_7=(x+b_3y+z)$, $L_8=(a_5x+b_4y+z)$, $L_9=(a_6x+b_5y+z)$, $L_{10}=(b_6y+z)$,
$L_{11}=(ax-bc^{2-r}y)$, $L_{12}=(x+y+z)$, $L_{13}=(a_5x+b_1y+z)$, $L_{14}=(a_6x+b_2y+z)$, and $L_{15}=(z)$. These two $3$ dimensional families
of $(3,5)$-nets appear in \cite{Stipins1}. We notice that both families of $(3,5)$-nets have members defined over $\Q$. For the case $M_1$, we
can make $b^2$ disappear from the equation by declaring $c=ad$ (the relations $a=1$ and $d=c$ are not allowed). Then, $b=
\frac{2da-1-da^2}{2da-2da^2-1+a-d^2a+d^2a^2}$, and it can be checked that for suitable $a,d \in \Z$ the conditions for being $(3,5)$-net are
satisfied.

\vspace{0.3 cm} \underline{\textbf{$(3,6)$-nets.}}

We have twelve main classes of Latin squares to check. The following is a list showing one member of each class. It was taken from \cite[pp.
129-137]{DeKe74}.
\begin{center}
\vspace{0.3 cm}{\tiny $M_1=$
\begin{tabular}{|cccccc}
\hline
   1 & 2 & 3 & 4 & 5 & 6\\
   2 & 3 & 4 & 5 & 6 & 1\\
   3 & 4 & 5 & 6 & 1 & 2\\
   4 & 5 & 6 & 1 & 2 & 3\\
   5 & 6 & 1 & 2 & 3 & 4\\
   6 & 1 & 2 & 3 & 4 & 5\\
\end{tabular}
$M_2=$
\begin{tabular}{|cccccc}
\hline
   1 & 2 & 3 & 4 & 5 & 6\\
   2 & 1 & 5 & 6 & 3 & 4\\
   3 & 6 & 1 & 5 & 4 & 2\\
   4 & 5 & 6 & 1 & 2 & 3\\
   5 & 4 & 2 & 3 & 6 & 1\\
   6 & 3 & 4 & 2 & 1 & 5\\
\end{tabular}
$M_3=$
\begin{tabular}{|cccccc}
\hline
   1 & 2 & 3 & 4 & 5 & 6\\
   2 & 3 & 1 & 5 & 6 & 4\\
   3 & 1 & 2 & 6 & 4 & 5\\
   4 & 6 & 5 & 2 & 1 & 3\\
   5 & 4 & 6 & 3 & 2 & 1\\
   6 & 5 & 4 & 1 & 3 & 2\\
\end{tabular}
$M_4=$
\begin{tabular}{|cccccc}
\hline
   1 & 2 & 3 & 4 & 5 & 6\\
   2 & 1 & 4 & 3 & 6 & 5\\
   3 & 4 & 5 & 6 & 1 & 2\\
   4 & 3 & 6 & 5 & 2 & 1\\
   5 & 6 & 1 & 2 & 4 & 3\\
   6 & 5 & 2 & 1 & 3 & 4\\
\end{tabular}}
\end{center}
\begin{center}
\vspace{0.3 cm}{\tiny $M_5=$
\begin{tabular}{|cccccc}
\hline
   1 & 2 & 3 & 4 & 5 & 6\\
   2 & 1 & 4 & 3 & 6 & 5\\
   3 & 4 & 5 & 6 & 1 & 2\\
   4 & 3 & 6 & 5 & 2 & 1\\
   5 & 6 & 2 & 1 & 4 & 3\\
   6 & 5 & 1 & 2 & 3 & 4\\
\end{tabular}
$M_6=$
\begin{tabular}{|cccccc}
\hline
   1 & 2 & 3 & 4 & 5 & 6\\
   2 & 1 & 4 & 5 & 6 & 3\\
   3 & 6 & 2 & 1 & 4 & 5\\
   4 & 5 & 6 & 2 & 3 & 1\\
   5 & 3 & 1 & 6 & 2 & 4\\
   6 & 4 & 5 & 3 & 1 & 2\\
\end{tabular}
$M_7=$
\begin{tabular}{|cccccc}
\hline
   1 & 2 & 3 & 4 & 5 & 6\\
   2 & 1 & 4 & 3 & 6 & 5\\
   3 & 5 & 1 & 6 & 4 & 2\\
   4 & 6 & 5 & 1 & 2 & 3\\
   5 & 3 & 6 & 2 & 1 & 4\\
   6 & 4 & 2 & 5 & 3 & 1\\
\end{tabular}
$M_8=$
\begin{tabular}{|cccccc}
\hline
   1 & 2 & 3 & 4 & 5 & 6\\
   2 & 1 & 6 & 5 & 3 & 4\\
   3 & 6 & 1 & 2 & 4 & 5\\
   4 & 5 & 2 & 1 & 6 & 3\\
   5 & 3 & 4 & 6 & 1 & 2\\
   6 & 4 & 5 & 3 & 2 & 1\\
\end{tabular}}
\end{center}
\begin{center}
\vspace{0.3 cm}{\tiny $M_9=$
\begin{tabular}{|cccccc}
\hline
   1 & 2 & 3 & 4 & 5 & 6\\
   2 & 3 & 1 & 6 & 4 & 5\\
   3 & 1 & 2 & 5 & 6 & 4\\
   4 & 6 & 5 & 1 & 2 & 3\\
   5 & 4 & 6 & 2 & 3 & 1\\
   6 & 5 & 4 & 3 & 1 & 2\\
\end{tabular}
$M_{10}=$
\begin{tabular}{|cccccc}
\hline
   1 & 2 & 3 & 4 & 5 & 6\\
   2 & 1 & 6 & 5 & 4 & 3\\
   3 & 5 & 1 & 2 & 6 & 4\\
   4 & 6 & 2 & 1 & 3 & 5\\
   5 & 3 & 4 & 6 & 2 & 1\\
   6 & 4 & 5 & 3 & 1 & 2\\
\end{tabular}
$M_{11}=$
\begin{tabular}{|cccccc}
\hline
   1 & 2 & 3 & 4 & 5 & 6\\
   2 & 1 & 4 & 5 & 6 & 3\\
   3 & 4 & 2 & 6 & 1 & 5\\
   4 & 5 & 6 & 2 & 3 & 1\\
   5 & 6 & 1 & 3 & 2 & 4\\
   6 & 3 & 5 & 1 & 4 & 2\\
\end{tabular}
$M_{12}=$
\begin{tabular}{|cccccc}
\hline
   1 & 2 & 3 & 4 & 5 & 6\\
   2 & 1 & 5 & 6 & 4 & 3\\
   3 & 5 & 4 & 2 & 6 & 1\\
   4 & 6 & 2 & 3 & 1 & 5\\
   5 & 4 & 6 & 1 & 3 & 2\\
   6 & 3 & 1 & 5 & 2 & 4\\
\end{tabular}}
\vspace{0.3 cm}
\end{center}

The Latin squares $M_1$ and $M_2$ correspond to the multiplication table of the groups $\Z/6\Z$ and $S_3$, respectively. The following is the
set up for the analysis of $(3,6)$-nets. We first fix one Latin square $M$ from the list above. Let $\A=\{ \A_1, \A_2, \A_3 \}$ be the
corresponding (possible) $(3,6)$-net, where $\A_1= \{L_1,\ldots,L_6 \}$, $\A_2= \{L_7,\ldots,L_{12}\}$,, and $\A_3= \{L_{13},\ldots,L_{18} \}$.
As before, we consider a new arrangement $\A'$ together with a point $P$ such that $\A'=\A \setminus \{L_1,L_7,L_{13} \}$, and $P=[[1,7,13]]\in
\X$. We label the lines of $\A'$ from $1$ to $15$ following the order of $\A$, i.e., $L'_1=L_2, \ldots, L'_5=L_6$, $L'_6=L_8$, etc, eliminating
$L_1$, $L_7$, and $L_{13}$. Let $L(\A',P)$ be the line in $\P^{13}$ for $(\A',P)$. The special lines (or points of $L(\A',P)$) $\alpha=L_1$,
$\beta=L_7$, and $\gamma=L_{13}$ are as $\alpha=[a_1:a_2:a_3:a_4:a_5:1:a_6:a_7:a_8:0:1:a_6:a_7:a_8]$,
$\beta=[1:b_1:b_2:b_3:0:b_4:b_5:b_6:b_7:b_8:1:b_1:b_2:b_3]$, and $\gamma=\gamma(c_1,c_2,...,c_8)$ depending on $M$. Since there is $[t,u] \in
\P^1$ satisfying $\alpha t+ \beta u = \gamma$, we can and do write $a_1$, $a_2$, $a_3$, $a_4$, $a_6$, $a_7$, $a_8$, $b_4$, $b_5$, $b_6$, and
$c_5$ with respect to the rest of the variables.

After that, we start imposing the points in $\X'$ which translates, as before, into $2\times 2$ determinants equal to zero. At this stage we
have $20$ equations given by these determinants, and $12$ variables. We choose appropriately from them to isolate variables so that they appear
with exponent $1$. In the way of solving these equations, we prove or disprove realization for $\A$. When the $(3,6)$-net exists, i.e. $\A$ is
realizable in $\P^2$ over some field, the equations for its lines can be taken as: $L_1=(y)$, $L_2=(a_1 x + y + z)$, $L_3=(a_2x+b_1y+z)$,
$L_4=(a_3x+b_2y+z)$, $L_5=(a_4x+ b_3y+z)$, $L_6=(a_5x+z)$, $L_7=(x)$, $L_8=(x+b_4y+z)$, $L_9=(a_6x+b_5y+z)$, $L_{10}=(a_7x+b_6y+z)$,
$L_{11}=(a_8 x + b_7 y+z)$, $L_{12}=(b_8y+z)$, $L_{13}=(ux-ty)$, $L_{14}=(x+y+z)$, $L_{15}=(a_6x+b_1y+z)$, $L_{16}=(a_7x+b_2y+z)$, $L_{17}=(a_8x
+ b_3y+z)$, and $L_{18}=(z)$, where $[t,u]$ satisfies $\alpha t+ \beta u = \gamma$.

Now we apply this procedure case by case. We first give the result, after that we indicate the order we solve the equations coming from the
points in $\X'$, and then we give a moduli parametrization whenever the net exits. For simplicity, we work always in characteristic zero. We
often omit the final expressions for the variables, although they can be given explicitly.

\vspace{0.2cm}

\underline{$M_1$}: ($\Z/6\Z$) This gives a three dimensional moduli space. We have that some of these nets can be defined over $\R$. We solve
the determinants in the following order: $[[4,6,15]]$ solve for $c_3$, $[[5,10,14]]$ solve for $c_8$, $[[1,9,15]]$ solve for $c_1$, $[[5,9,13]]$
solve for $c_7$, $[[3,10,12]]$ solve for $c_6$, $[[2,10,11]]$ solve for $b_3$, $[[3,9,11]]$ solve for $c_2$, and $[[2,8,15]]$ solve for $b_2$.
If $a_5=a$, $b_1=d$, $b_8=b$, and $c_4=c$, then they must satisfy: \begin{center}
$c^2(-1+a)b^4(a^2-a^2b+cab+ab-2a+ca-bc)-b^2c(2c^2b^2+5cab+4a^2b^2c-4ca^2b-2b^2a^3-a^2c+3a^2-2a^3-5a^2b+ca^3+2a^2b^2-bc^2a+4ba^3-4ac^2b^2-3ab^2c
-a^3b^2c+c^2ba^2+2a^2b^2c^2)d+(bc+a-ab)(a^2b^2c^2+c^2b^2-2ac^2b^2+a^2b^2c-ab^2c+2cab-ca^2b+a^2b^2-2a^2b+a^2)d^2=0.$
\end{center} So, the moduli space for these nets is an open set of this hypersurface.

\underline{$M_2$}: ($S_3$) This gives a three dimensional moduli space parametrized by an open set of $\Af^3$. It does not contains $(3,6)$-nets
defined over $\R$. The reason is that we need the square root of $-1$ to define the nets. Moreover, all of them have extra $3$-points, apart
from the ones coming from $\X$. The order we take is: $[[5,10,14]]$ solve for $c_8$, $[[2,6,15]]$ solve for $c_1$, $[[1,10,13]]$ solve for
$c_7$, $[[1,9,12]]$ solve for $c_6$, $[[2,10,11]]$ solve for $b_1$, $[[5,6,12]]$ solve for $c_3$, $[[1,8,15]]$ solve for $b_2$, $[[1,7,14]]$
solve for $b_8$, and $[[2,8,14]]$ solve for $c_2$. If $i=\sqrt{-1}$, $a_5=a$, $b_3=e$ and $c_4=c$, then the expressions for the variables are:

\vspace{0.2cm}
\begin{center}
$a_1=\frac{(1+i\sqrt{3})(2c+e-i\sqrt{3}e)a}{4ce}$ $a_2= \frac{2ace}{2aec-ac-ce-ice\sqrt{3}+ae+iae\sqrt{3}+ica\sqrt{3}}$
$a_3=\frac{(-1+i\sqrt{3})(ae-iae\sqrt{3}-2ce+2ac)a}{2(2ae-2ce+2aec+ac+ica\sqrt{3})}$ $a_4=a$ $a_5=a$
$a_6=\frac{(-1+i\sqrt{3})(ae-iae\sqrt{3}-2ce+2ac)a}{2(2aec-ac-ce-ice\sqrt{3}+ae+iae\sqrt{3}+ica\sqrt{3})}$ $a_7=
\frac{(1+i\sqrt{3})(e-i\sqrt{3}e+2c)a^2}{2(2ae-2ce+2aec+ac+ica\sqrt{3})}$ $a_8=\frac{(1+i\sqrt{3})a}{2}$

$b_1= \frac{(-1+i\sqrt{3})e^2(a-c)}{2aec-ac-ce-ice\sqrt{3}+ae+iae\sqrt{3}+ica\sqrt{3}}$
$b_2=\frac{(1+i\sqrt{3})(-ce+ae+ac)e}{2ae-2ce+2aec+ac+ica\sqrt{3}}$ $b_3= e$ $b_4= \frac{(-1+i\sqrt{3})(a-c)e}{2ac}$ $b_5=
\frac{(1+i\sqrt{3})(-ce+ae+ac)e}{2aec-ac-ce-ice\sqrt{3}+ae+iae\sqrt{3}+ica\sqrt{3}}$ $b_6=\frac{2aec}{2ae-2ce+2aec+ac+ica\sqrt{3}}$ $b_7=
\frac{(1+i\sqrt{3})e}{2}$ $b_8= \frac{(1+i\sqrt{3})e}{2}$
\end{center}
\vspace{0.3cm}

For instance, if we plug in $a= \frac{c+ic\sqrt{3}}{1+i\sqrt{3}-2c}$ and $e= \frac{c(1+i\sqrt{3})}{2(c-1)}$, we get a one dimensional family of
arrangements of $18$ lines with $t_2=18$, $t_3=39$, $t_4=3$, $t_k=0$ otherwise.

\underline{$M_3$}: This gives a three dimensional moduli space which does not contains $(3,6)$-nets defined over $\R$. The reason again is that
we need to have the square root of $-1$ to realize the nets. The order we solve is: $[[5,10,11]]$ solve for $b_8$, $[[1,9,15]]$ solve for $c_8$,
$[[5,9,12]]$ solve for $c_6$, $[[3,6,15]]$ solve for $c_1$, $[[1,10,13]]$ solve for $c_7$, $[[4,9,11]]$ solve for $b_3$, $[[1,6,12]]$ solve for
$b_1$, and $[[3,10,12]]$ solve for $b_2$. If $a_5=a$, $c_3=d$, $c_2=e$, and $c_4=c$, then they must satisfy: \begin{center}
$(e^2a^2+e^2-e^2a-2a^2de-de+d^2+3dea+d^2a^2-2d^2a)+(-ea-e+ad-d)c+c^2=0$
\end{center}and so its moduli space is an open set of this hypersurface. Moreover, by solving for $c$, we have that: $c= \frac{1}{2}(ea+e-ad+d \pm
\sqrt{-3}(a-1)(d-e))$. But, we cannot have $a=1$ or $d=e$, and so this shows that the square root of $-1$ is necessary.

\underline{$M_4$}: This case is not possible over $\C$. To get the contradiction, we take: $[[5,10,13]]$ solve for $c_7$, $[[3,7,15]]$ solve for
$c_6$, $[[2,8,15]]$ solve for $b_2$, $[[4,6,15]]$ solve for $c_3$, $[[5,6,14]]$ solve for $a_5$, $[[1,9,15]]$ solve for $c_8$, $[[1,10,14]]$
solve for $c_1$, $[[3,8,14]]$ solve for $c_2$, $[[2,10,11]]$ solve for $c_4$, and $[[2,6,13]]$ solve for $b_1$. At this stage, we obtain several
possibilities from the equation given by $[[2,6,13]]$, none of them possible (for example, $a_2=a_6$).

\underline{$M_5$}: This case is not possible over $\C$. By solving $[[5,10,13]]$ for $c_7$, and then $[[3,7,15]]$ for $c_6$, we obtain $a_6=a_7$
which is a contradiction.

\underline{$M_6$}: This gives a two dimensional moduli space, and so this parameter space are not always three dimensional (see \cite[p.
14]{Stipins1}). Some of these nets can be defined over $\R$. The order we take is: $[[5,10,11]]$ solve for $a_5$, $[[1,9,15]]$ solve for $c_8$,
$[[3,7,15]]$ solve for $c_6$, $[[2,6,15]]$ solve for $b_1$, $[[5,6,13]]$ solve for $c_7$, $[[4,9,11]]$ solve for $b_3$, $[[2,9,13]]$ solve for
$c_1$, $[[1,10,12]]$ solve for $c_3$, and $[[3,9,12]]$ solve for $b_2$. If $b_8=b$, $c_2=d$, and $c_4=c$, then they must satisfy:
$$bc(1-c)(bc-c-b)+(bc^3+b^2-5bc^2+3bc-2b^2c+b^2c^2-c^3+2c^2)d+(-b+2bc-2c+c^2)d^2=0.$$ Thus, its moduli space is an open set of this
hypersurface.

\underline{$M_7$}: This gives a two dimensional moduli space parametrized by an open set of $\Af^2$. These nets can be defined over $\Q$. The
order we solve is the following: $[[5,6,13]]$ solve for $c_7$, $[[3,6,15]]$ solve for $b_2$, $[[1,9,15]]$ solve for $c_8$, $[[5,9,12]]$ solve
for $c_6$, $[[1,10,14]]$ solve for $b_3$, $[[3,9,11]]$ solve for $b_8$, $[[4,8,11]]$ solve for $c_3$, $[[4,10,13]]$ solve for $c_2$,
$[[4,7,15]]$ solve for $b_1$, and $[[5,7,11]]$ solve for $c_1$. If $a_5=a$ and $c_4=c$, then we have:

\vspace{0.4cm}
\begin{center}
$a_1=\frac{(c^2-4c+2ac+4-2a)a}{c(a-2)(c-2)}$ $a_2= \frac{(c-1)(c-2)(a-2)a}{a^2c^2+a^2-2a^2c-2c^2a+5ac-2a+c^2-2c}$
$a_3=\frac{ac(a+c-2)}{-c^2-ac+c^2a+2c-2a+a^2}$ $a_4=\frac{(a-2)(a-ac+c-2)a}{-c^2+c^2a-3ac+2c+a^2c+2a-a^2}$ $a_5=a$
$a_6=\frac{(a+c-2)(-a+ac-c+2)a}{a^2c^2+a^2-2a^2c-2c^2a+5ac-2a+c^2-2c}$ $a_7= \frac{(a-2)a^2(c-1)}{c^2+ac-c^2a-2c+2a-a^2}$
$a_8=\frac{a(c^2-4c+2ac+4-2a)}{-c^2+c^2a-3ac+2c+a^2c+2a-a^2}$

$b_1= \frac{(a-1)(a-2)(c-2)^2c}{(a+c-2)(a^2c^2+a^2-2a^2c-2c^2a+5ac-2a+c^2-2c)}$ $b_2=\frac{(c-a)(a-2)(c-2)}{-c^2-ac+c^2a+2c-2a+a^2}$ $b_3=
\frac{(a-2)^2(c-1)a(c-2)}{(a+c-2)(c^2-c^2a+3ac-2c-a^2c-2a+a^2)}$ $b_4= \frac{(c-a)(a-2)(c-2)}{ac(a+c-2)}$ $b_5=
\frac{(c-1)(c-2)(a-2)a}{a^2c^2+a^2-2a^2c-2c^2a+5ac-2a+c^2-2c}$ $b_6=\frac{c(a-2)(c-2)a(a-1)}{(c^2+ac-c^2a-2c+2a-a^2)(a+c-2)}$ $b_7=
\frac{c(a-2)(c-2)}{-c^2+c^2a-3ac+2c+a^2c+2a-a^2}$ $b_8= \frac{(a-2)(c-2)}{2-a-c}$
\end{center}
\vspace{0.4cm}

\underline{$M_8$}: This also gives a two dimensional moduli space. Some of these nets can be defined over $\R$. The order we solve is the
following: $[[2,6,15]]$ solve for $b_1$, $[[1,10,13]]$ solve for $c_7$, $[[1,7,15]]$ solve for $c_6$, $[[5,7,14]]$ solve for $c_8$,
$[[4,10,11]]$ solve for $c_3$, $[[5,6,13]]$ solve for $b_2$, $[[2,10,14]]$ solve for $b_3$, $[[5,9,11]]$ solve for $a_5$, and $[[3,7,11]]$ solve
for $c_1$. If $b_8=b$, $c_4=c$, and $c_2=e$, then they have to satisfy:
$$c^2(c-b)(4c^2-6cb-b^3+3b^2)+c(cb-2c+b)(6c^2-9cb-b^3+4b^2)e+ (bc-b+c)(cb-2c+b)^2e^2=0.$$

Thus, its moduli space is an open set of this hypersurface.

\underline{$M_9$}: This gives a three dimensional moduli space. Some of them can be defined over $\R$. The order we solve is the following:
$[[5,10,11]]$ solve for $a_5$, $[[1,10,14]]$ solve for $c_8$, $[[4,7,15]]$ solve for $c_6$, $[[4,9,12]]$ solve for $b_3$, $[[1,8,15]]$ solve for
$c_7$, $[[5,8,12]]$ solve for $c_2$, $[[5,6,14]]$ solve for $c_1$, and $[[3,6,15]]$ solve for $b_2$. If $b_1=e$, $b_8=b$, $c_4=c$, and $c_3=d$,
then they have to satisfy: \begin{center}$(b^2c^2+c^2+bc-b^2c-2bc^2)+(-2c+2bc+ce-bec+e^2b-eb)d+(-e+1)d^2=0.$\end{center}

Thus, its moduli space is an open set of this hypersurface.

\underline{$M_{10}$}: This gives a two dimensional moduli space. Some of these nets can be defined over $\R$. The order we solve is the
following: $[[5,10,11]]$ solve for $a_5$, $[[1,7,15]]$ solve for $b_1$, $[[1,10,12]]$ solve for $c_6$, $[[3,6,15]]$ solve for $b_2$,
$[[5,6,13]]$ solve for $c_7$, $[[5,7,14]]$ solve for $c_8$, $[[4,8,15]]$ solve for $b_3$, $[[3,7,11]]$ solve for $c_3$, and $[[2,8,11]]$ solve
for $c_2$. If $b_8=b$, $c_4=c$, and $c_1=e$, then they have to satisfy: \begin{center}$ce(c-2e)+(2ce-c-e)(e-c)b+c(1-e)(e-c)b^2=0.$\end{center}

Thus, its moduli space is an open set of this hypersurface.

\underline{$M_{11}$}: This also gives a two dimensional moduli space. Some of these nets can be defined over $\R$. The order we solve is the
following: $[[5,10,11]]$ solve for $a_5$, $[[1,9,15]]$ solve for $c_8$, $[[3,8,11]]$ solve for $c_7$, $[[3,7,15]]$ solve for $c_6$, $[[4,6,15]]$
solve for $b_3$, $[[2,8,15]]$ solve for $b_1$, $[[4,9,11]]$ solve for $c_2$, $[[5,7,14]]$ solve for $c_3$, and $[[1,8,14]]$ solve for $c_1$. An
extra property for this nets is that $c_7$ has to be zero, and so $L_{13}$, $L_{16}$, and $L_{18}$ have always a common point of incidence. If
$b_2=e$, $b_8=b$, and $c_4=c$, then they must satisfy: \begin{center}$c(b-1)(bc-b-c)+(b^2c-2bc+c-b^2+2b)e-e^2=0.$\end{center}

Thus, its moduli space is an open set of this hypersurface.

\underline{$M_{12}$}: This case is not possible over $\C$. To achieve contradiction, we take: $[[2,9,15]]$ solve for $c_8$, $[[5,10,13]]$ solve
for $c_7$, $[[3,6,15]]$ solve for $b_2$, $[[1,8,15]]$ solve for $c_3$, $[[1,10,12]]$ solve for $c_6$, $[[5,9,11]]$ solve for $c_2$, and
$[[5,6,12]]$ solve for $b_1$. Then, the equation induced by $[[1,9,13]]$ gives six possibilities, none of them is possible.

\vspace{0.3cm}

\underline{\textbf{$(3,8)$-nets corresponding to the Quaternion group.}}

We now compute the $(3,8)$-nets corresponding to the multiplication table of the Quaternion group.

\begin{center}
\vspace{0.3 cm}{\tiny $M=$
\begin{tabular}{|cccccccc}
\hline
   1 & 2 & 3 & 4 & 5 & 6 & 7 & 8 \\
   2 & 1 & 6 & 7 & 8 & 3 & 4 & 5 \\
   3 & 6 & 2 & 5 & 7 & 1 & 8 & 4 \\
   4 & 7 & 8 & 2 & 3 & 5 & 1 & 6 \\
   5 & 8 & 7 & 6 & 2 & 4 & 3 & 1 \\
   6 & 3 & 1 & 8 & 4 & 2 & 5 & 7 \\
   7 & 4 & 5 & 1 & 6 & 8 & 2 & 3 \\
   8 & 5 & 4 & 3 & 1 & 7 & 6 & 2
\end{tabular}}
\end{center}
\vspace{0.3 cm}

In this case, we have a three dimensional moduli space for them, given by an open set of $\Af^3$. Also, these $(3,8)$-nets can be defined over
$\Q$ (so we can even draw them). This example shows again that non-abelian groups can also realize nets over $\C$. The set up is similar to what
we did before. In this case, $\A'=\A \setminus \{ L_1, L_9, L_{17} \}$ and $P=[[1,9,17]]$. Our distinguished points on $L(\A',P) \subseteq
\P^{19}$ are: $\alpha=[a_1:a_2:a_3:a_4:a_5:a_6:a_7:1:a_8:a_9:a_{10}:a_{11}:a_{12}:0:1:a_8:a_9:a_{10}:a_{11}:a_{12}]$,
$\beta=[1:b_1:b_2:b_3:b_4:b_5:0:b_6:b_7:b_8:b_9:b_{10}:b_{11}:b_{12}:1:b_1:b_2:b_3:b_4:b_5]$, and
$\gamma=[1:c_1:c_2:c_3:c_4:c_5:c_6:1:c_4:c_5:c_6:c_1:c_2:c_3:c_7:c_8:c_9:c_{10}:c_{11}:c_{12}]$. Let $[t:u] \in \P^1$ such that $\alpha t+ \beta
u = \gamma$. We isolate first $a_1$, $a_2$, $a_3$, $a_4$, $a_5$, $a_6$, $a_8$, $a_9$, $a_{10}$, $a_{11}$, $b_6$, $b_7$, $b_8$, $b_9$, $b_{10}$,
$b_{11}$, and $c_7$ with respect to the other variables. The following is the order we solve (some of) the $2 \times 2$ determinants given by
the $3$-points in $\X'$: $[[1,11,21]]$ solve for $c_{10}$, $[[2,10,21]]$ solve for $c_9$, $[[3,12,21]]$ solve for $c_{11}$, $[[4,8,21]]$ solve
for $b_3$, $[[5,13,21]]$ solve for $c_8$, $[[7,14,15]]$ solve for $b_{12}$, $[[5,14,20]]$ solve for $c_4$, $[[2,14,17]]$ solve for $c_1$,
$[[4,9,20]]$ solve for $c_2$, $[[6,13,15]]$ solve for $c_5$, $[[3,8,20]]$ solve for $b_5$, $[[3,10,15]]$ solve for $b_4$, $[[3,9,18]]$ solve for
$c_6$, and $[[3,11,19]]$ solve for $b_1$. Then, if we write $a_7=a$, $b_2=e$, and $c_3=d$, the expressions for all the variables are:

\vspace{0.4cm}
\begin{center}
$a_1=\frac{ad-a-d}{a-2} \ \ \ a_2=\frac{2e^2d-2ed+ed^2-e^2d^2+(-2ed^2+e^2d^2+2e+6ed-3e^2d-4)a+(-4ed-e+4+ed^2+e^2d)a^2}{(ae-2a-2e+2)(ade-ed+d-a-da)}$ \\
$a_3=\frac{e(ade-ed+d-a-da)}{ae-2a-2e+2} \ \ \ a_4=d$ \\
$a_5=\frac{4d+2e^2d-6ed+e^2d^2-ed^2+(-2e^2d^2+2ed^2-8d-e^2d+10ed-2e)a+(4d+e+e^2d^2-ed^2-4ed)a^2}{(ae-2a-2e+2)(a+d-ad-de)}$\\
$a_6=\frac{(a+d-ad-de)(ae+dae-4a-2e-ed+4)}{(ae-2a-2e+2)(ade-a-da-2ed+2+d)} \ \ \ a_7=a \ \ \ a_8=\frac{(a+d-da-2)e}{ae-2a-2e+2}$ \\
$a_9=\frac{2e^2d-2ed+ed^2-e^2d^2+(-2ed^2+e^2d^2+2e+6ed-3e^2d-4)a+(-4ed-e+4+ed^2+e^2d)a^2}{(a+d-ad-2)(ae-2a-2e+2)}$\\
$a_{10}=a+d-ad  \ \ \ a_{11}=\frac{ade+ae-4a-2e-ed+4}{ae-2a-2e+2}$ \\
$a_{12}=\frac{2e^2d+4d-6ed+e^2d^2-ed^2+(-2e^2d^2+2ed^2-8d-e^2d+10ed-2e)a+(4d+e+e^2d^2-ed^2-4ed)a^2}{(ade-a-da-2ed+d+2)(ae-2a-2e+2)}$
\end{center}

\begin{center}
$b_1=\frac{-2e+ed+ae-2a}{-ed+dae+d-a-da} \ \ \ b_2=e \ \ \ b_3=2 \ \ \ b_4=\frac{ae-2a-2e-ed+4}{a+d-ad-de} \ \ \
b_5=\frac{ade+ae-4a-2e-ed+4}{ade-a-da-2ed+d+2} \ \ \ b_6=\frac{2}{d}$ $b_7=\frac{e(ad-d+2-a)}{ad+de-a-d} \ \ \
b_8=-\frac{-2e+ed+ae-2a}{a+d-ad-2} \ \ \ b_9= 2-a \ \ \ b_{10}=\frac{ae+dae-4a-ed-2e+4}{ade-ed+d-a-da}$
$b_{11}=\frac{ae-2a-2e+4-ed}{ade-a-da-2ed+d+2} \ \ \ b_{12}=\frac{a}{a-1}$,
\end{center} with $[t:u]=[2-a:d(a-1)] \in \P^1$.
\vspace{0.4cm}

Since $b_3=2$, these $(3,8)$-nets are not possible in characteristic $2$. The lines for these $(3,8)$-nets can be written as: $L_1=(y)$,
$L_2=(a_1 x + y + z)$, $L_3=(a_2x+b_1y+z)$, $L_4=(a_3x+b_2y+z)$, $L_5=(a_4x+ b_3y+z)$, $L_6=(a_5x+b_4 y+z)$, $L_7=(a_6x+b_5y+z)$,
$L_8=(a_7x+z)$, $L_9=(x)$, $L_{10}=(x+b_6y+z)$, $L_{11}=(a_8 x + b_7 y+z)$, $L_{12}=(a_9x+b_8y+z)$, $L_{13}=(a_{10}x+b_9y+z)$,
$L_{14}=(a_{11}x+b_{10}y+z)$, $L_{15}=(a_{12}x+b_{11}y+z)$, $L_{16}=(b_{12}y+z)$, $L_{17}=(ux-ty)$, $L_{18}=(x+y+z)$, $L_{19}=(a_8x+b_1y+z)$,
$L_{20}=(a_9x+b_2y+z)$, $L_{21}=(a_{10}x+b_3y+z)$, $L_{22}=(a_{11}x+b_4y+z)$, $L_{23}=(a_{12}x+b_5y+z)$, and $L_{24}=(z)$.
\vspace{0.4cm}

A natural question, which we leave open, is the following:
\begin{question}
\textit{Is there a combinatorial characterization of the main classes of $q \times q$ Latin squares which realize $(3,q)$-nets in $\P_\C^2$?}
\label{q1}
\end{question}

\bibliographystyle{amsplain}

%-----------------------------------------------------------------------------------------------------------------------------------------------

\vspace{0.1 cm} {\small Department of Mathematics and Statistics, University of Massachusetts at Amherst, USA.}
%{\tiny MSC classes: primary 52C30 - 14J10 - 05B15, secondary 05B30-14H50-52C35}
\end{document}